\documentclass[letterpaper, 10 pt, conference]{ieeeconf}

\IEEEoverridecommandlockouts                             
\let\proof\relax 
\let\endproof\relax
\usepackage{amsthm}
\usepackage{amsfonts}
\usepackage{amssymb,bm}
\usepackage{mathrsfs}
\usepackage{mathtools}
\usepackage{amsmath}

\usepackage{url}
\usepackage{graphicx}
\usepackage{textcomp}
\usepackage{comment}

\usepackage{xcolor}

\newtheorem{theorem}{Theorem}
\newtheorem{definition}{Definition}

\newtheorem{remark}{Remark}

\title{\LARGE \bf 
Counter-Adversarial Learning with Inverse Unscented Kalman Filter
}

\author{Himali Singh, Kumar Vijay Mishra and Arpan Chattopadhyay
\thanks{$^\ast$K. V. M. and A. C. have made equal contributions.}%
\thanks{H. S. and A. C. are with the Electrical Engineering Department, Indian Institute of Technology Delhi, India. {\tt\small\{eez208426,arpanc\}@ee.iitd.ac.in}} 
\thanks{K. V. M. is with the United States DEVCOM Army Research Laboratory, Adelphi, MD 20783 USA. {\tt\small kvm@ieee.org}}%
\thanks{A. C. acknowledges support via the faculty seed grant, professional development fund and professional development  allowance from IIT Delhi, and the  seed grant and grant no. RP04215G from I-Hub Foundation for Cobotics. H. S. acknowledges support via Prime Minister Research Fellowship. K. V. M. acknowledges support from the National Academies of Sciences, Engineering, and Medicine via Army Research Laboratory Harry Diamond Distinguished Fellowship.}%
}

\begin{document}

\maketitle
\thispagestyle{empty}
\pagestyle{empty}

\begin{abstract}
In counter-adversarial systems, to infer the strategy of an intelligent adversarial agent, the defender agent needs to cognitively sense the information that the adversary has gathered about the latter. Prior works on the problem employ linear Gaussian state-space models and solve this inverse cognition problem by designing inverse stochastic filters. However, in practice, counter-adversarial systems are generally highly nonlinear. In this paper, we address this scenario by formulating inverse cognition as a nonlinear Gaussian state-space model, wherein the adversary employs an unscented Kalman filter (UKF) to estimate the defender's state with reduced linearization errors. To estimate the adversary's estimate of the defender, we propose and develop an \textit{inverse UKF} (IUKF) system. We then derive theoretical guarantees for the stochastic stability of IUKF in the mean-squared boundedness sense. Numerical experiments for multiple practical applications show that the estimation error of IUKF converges and closely follows the \textit{recursive} Cram\'{e}r-Rao lower bound.
\end{abstract}

\section{Introduction}
Complex and dynamic environments in many engineering applications require intelligent and autonomous agents that \textit{cognitively} sense their environment, acquire the relevant information, and then use it to adapt in real-time to environmental changes for an improved performance \cite{haykin2006cognitive,guerci2010cognitive}. In this context, \textit{inverse cognition} --- wherein a \textit{defender} agent learns the information about itself sensed by a cognitive \textit{attacker} or \textit{adversarial} agent --- has recently gathered significant research interest \cite{krishnamurthy2020identifying}. The problem is motivated by the need to design counter-autonomous adversarial systems. For instance, a cognitive radar estimates its target's kinematic state and then adapts its waveform and processing for enhanced target detection \cite{mishra2017performance} and tracking \cite{bell2015cognitive}. An intelligent target observes the radar's adaptive behavior and predicts the latter's future actions in a Bayesian sense. Similar examples abound in interactive learning, autonomous sensor calibration, fault diagnosis, and cyber-physical security \cite{mattila2020hmm,krishnamurthy2019how}.

In inverse cognition, it is imperative to first identify if the adversary is cognitive. To this end, \cite{krishnamurthy2020identifying} developed stochastic revealed preferences-based algorithms to ascertain if the adversary optimizes a utility function and if so, estimate that function. Further, \cite{krishnamurthy2019how} modeled this problem as an inverse Bayesian filtering problem. A Bayesian filter provides a posterior distribution for an underlying state given its noisy observations. Its so-called inverse then reconstructs this posterior distribution given the actual state and noisy measurements of the posterior \cite{krishnamurthy2019how}. These types of inverse problems may be traced to \cite{kalman1964linear} to find the cost criterion for a given control policy. Later, analogous formulations appeared in revealed preferences problem of microeconomics theory \cite{varian1992microeconomic} 
and inverse reinforcement learning (IRL) \cite{choi2012nonparametric,levine2011nonlinear}. In particular, the inverse cognition framework is a generalization of IRL, where the adversary's reward function is learned passively \cite{ng2000algorithms}, to an active learning application.

Among the Bayesian filters for inverse cognition, \cite{mattila2020hmm} proposed \textit{inverse hidden Markov model} to estimate the attacker's observations and observation likelihood given noisy measurements of its posterior. This was later extended to a linear Gaussian state-space model \cite{krishnamurthy2019how}, where the attacker employed Kalman filter (KF) to estimate the defender's state and the latter estimated the \textit{former's estimate of the defender} using an inverse Kalman filter (IKF). However, in practice, counter-adversarial systems are highly non-linear. In our recent work \cite{singh2022inverse}, we proposed an inverse extended Kalman filter (IEKF) for non-linear system dynamics. Later, we developed stability guarantees for IEKF and its variants in \cite{singh2022inverse_part1}. Further, in \cite{singh2022inverse_part2}, we addressed the case of unknown forward filter and system using a new reproducing kernel Hilbert space (RKHS)-based EKF to jointly learn the system parameters and then estimate the state.

While extended Kalman filter (EKF) is a popular choice for non-linear filtering applications, difficulties in its online implementation and unreliable performance have led to the development of unscented Kalman filter (UKF) \cite{julier2004unscented}. Contrary to the EKF that linearizes the system model at the state estimates and often fails when high non-linearities are present, the UKF employs the unscented transform \cite{julier1995new} and propagates a deterministic set of `sigma points' through the non-linear system. By avoiding computation of the Jacobian matrices, the UKF approximates a Gaussian random variable's posterior mean and covariance under non-linear transformation accurately up to the third-order for any non-linearity. Note that the EKF achieves only first-order accuracy with the same computational complexity \cite{wan2000unscented}. The sufficient conditions for the stochastic stability of UKF have been discussed in the literature for only linear measurements \cite{xiong2006performance_ukf}, where the estimation error was exponentially bounded in a mean-squared sense subject to mild system conditions. 

In this paper, given the widespread use and superior performance of UKF, we propose and develop inverse UKF (IUKF) that a counter-adversarial system may employ to infer the adversary's estimate in a general non-linear setting. 
In this framework, the adversary employs UKF as the forward Bayesian filter to estimate the defender's state. Note that the IUKF is different from the \textit{inversion of UKF} \cite{zhengyu2021iterated}, which estimates the input based on the output. Clearly, this inversion of the UKF does not necessarily take the same mathematical form as the UKF, is employed on the adversary's side, and is inapplicable to our inverse cognition problem. Our IUKF is a different formulation that is focused on estimating the inference of an adversary that also employs a UKF to estimate the defender's state. 


We also derive sufficient conditions under which the proposed IUKF achieves stochastic stability. 
In the process, we also obtain hitherto unreported stochastic stability results for the forward UKF under non-linear measurements. The IUKF's error dynamics depend on the set of sigma points chosen by the forward filter and, therefore, the derivation of these theoretical guarantees is non-trivial, especially the bounds on the Jacobian matrices of IUKF's functions. We validate our models and methods for various non-linear systems through numerical experiments 
using recursive Cram\'{e}r-Rao lower bound (RCRLB) \cite{tichavsky1998posterior} as a benchmark.

Throughout the paper, we reserve boldface lowercase and uppercase letters for vectors (column vectors) and matrices, respectively. The notation $[\mathbf{a}]_{i}$ is used to denote the $i$-th component of vector $\mathbf{a}$ and $[\mathbf{A}]_{i,j}$ denotes the $(i,j)$-th component of matrix $\mathbf{A}$, with $[\mathbf{A}]_{(:,j)}$ denoting the $j$-th column of the matrix. The transpose operation is $(\cdot)^{T}$; the $l_{2}$ norm of a vector is $\|\cdot\|_{2}$; and the notation $\textrm{Tr}(\mathbf{A})$ and $\|\mathbf{A}\|$, respectively, denote the trace and spectral norm of $\mathbf{A}$. For matrices $\mathbf{A}$ and $\mathbf{B}$, the inequality $\mathbf{A}\preceq\mathbf{B}$ means that $\mathbf{B}-\mathbf{A}$ is a positive semidefinite (p.s.d.) matrix. For a function $f:\mathbb{R}^{n}\rightarrow\mathbb{R}^{m}$, $\frac{\partial f}{\partial\mathbf{x}}$ denotes the $\mathbb{R}^{m\times n}$ Jacobian matrix. Also, $\mathbf{I}_{n}$ and $\mathbf{0}_{n\times m}$ denote a `$n\times n$' identity matrix and a `$n\times m$' all zero matrix, respectively. We denote the Cholesky decomposition of matrix $\mathbf{A}$ as $\mathbf{A}=\sqrt{\mathbf{A}}\sqrt{\mathbf{A}}^{T}$.

\section{System Model}
\label{sec:inverse UKF}
Consider the defender's discrete-time state process $\lbrace\mathbf{x}_{k}\rbrace_{k\geq 0}$, where $\mathbf{x}_{k}\in\mathbb{R}^{n_{x}\times 1}$, evolving as
\par\noindent\small
\begin{align}
    \mathbf{x}_{k+1}=f(\mathbf{x}_{k})+\mathbf{w}_{k},\label{eqn:state evolution x}
\end{align}
\normalsize
where $\mathbf{w}_{k}\sim\mathcal{N}(\mathbf{0}_{n_{x}\times 1}, \mathbf{Q})$ represents the process noise with covariance matrix $\mathbf{Q}$. The defender perfectly knows its current state. At the time $k$, the adversary's observation is
\par\noindent\small
\begin{align}
    \mathbf{y}_{k}=h(\mathbf{x}_{k})+\mathbf{v}_{k}\;\in\mathbb{R}^{n_{y}\times 1},\label{eqn:observation y}
\end{align}
\normalsize
where $\mathbf{v}_{k}\sim\mathcal{N}(\mathbf{0}_{n_{y}\times 1},\mathbf{R})$ is the adversary's measurement noise with covariance matrix $\mathbf{R}$. The adversary infers an estimate $\hat{\mathbf{x}}_{k}$ of defender's current state $\mathbf{x}_{k}$ using the observations $\lbrace\mathbf{y}_{j}\rbrace_{1\leq j\leq k}$ with its forward UKF. This estimate is then used by the adversary to compute and take an action $g(\hat{\mathbf{x}}_{k})$ whose noisy observation $\mathbf{a}_{k}\in\mathbb{R}^{n_{a}\times 1}$ recorded by the defender is
\par\noindent\small
\begin{align}
    \mathbf{a}_{k}=g(\hat{\mathbf{x}}_{k})+\bm{\epsilon}_{k},\label{eqn:observation a}
\end{align}
\normalsize
where $\bm{\epsilon}_{k}\sim\mathcal{N}(\mathbf{0}_{n_{a}\times 1},\bm{\Sigma}_{\epsilon})$ is defender's measurement noise with covariance matrix $\bm{\Sigma}_{\epsilon}$. Finally, defender computes the estimate $\Hat{\Hat{\mathbf{x}}}_{k}$ of $\hat{\mathbf{x}}_{k}$ using $\lbrace\mathbf{a}_{j},\mathbf{x}_{j}\rbrace_{1\leq j\leq k}$ in IUKF.

The functions $f(\cdot)$, $h(\cdot)$, and $g(\cdot)$ are suitable non-linear vector-valued functions. The noise processes $\lbrace\mathbf{w}_{k}\rbrace_{k\geq 0}$, $\lbrace\mathbf{v}_{k}\rbrace_{k\geq 1}$, and $\lbrace\bm{\epsilon}_{k}\rbrace_{k\geq 1}$ are mutually independent and identically distributed across time. Throughout this paper, the system model is perfectly known to both the defender and adversary. This is the same as the assumption made in prior formulations of IKF \cite{krishnamurthy2019how} and IEKF \cite{singh2022inverse}. Furthermore, our numerical experiments later show that the proposed IUKF is observed to provide reasonable estimates even when the forward UKF assumption does not hold.

\noindent\textbf{Forward UKF:} The state estimates in the UKF are obtained as a weighted sum of a set of $2n_{x}+1$ sigma points, generated deterministically from the previous state estimate. Considering scaling parameter $\kappa\in\mathbb{R}$, the sigma points $\lbrace\widetilde{\mathbf{x}}_{i}\rbrace_{0\leq i\leq 2n_{x}}$ generated from state estimate $\hat{x}$ and its error covariance matrix estimate $\bm{\Sigma}$ are
\par\noindent\small
\begin{align}
  &\widetilde{\mathbf{x}}_{i}=S_{gen}(\hat{\mathbf{x}},\bm{\Sigma})\nonumber\\
  &=\begin{cases}
  \hat{\mathbf{x}},\;\;\;\;\;\;\;i=0,\\
 \hat{\mathbf{x}}+\left[\sqrt{(n_{x}+\kappa)\bm{\Sigma}}\right]_{(:,i)},\;i=1,2,\hdots,n_{x}\\
    \hat{\mathbf{x}}-\left[\sqrt{(n_{x}+\kappa)\bm{\Sigma}}\right]_{(:,i-n_{x})},\;i=n_{x}+1,n_{x}+2,\hdots,2n_{x}\end{cases},\label{eqn:sigma points generation}
\end{align}
\normalsize
with their weights $\omega_{i}=\begin{cases}\frac{\kappa}{n_{x}+\kappa} & i=0\\\frac{1}{2(n_{x}+\kappa)} & i=1,2,\hdots,2n_{x}\end{cases}$. Denote the sigma points generated and propagated for the time update by $\lbrace\mathbf{s}_{i,k}\rbrace_{0\leq i\leq 2n_{x}}$ and $\lbrace\mathbf{s}^{*}_{i,k+1|k}\rbrace_{0\leq i\leq 2n_{x}}$, respectively. Similarly, $\lbrace\mathbf{q}_{i,k+1|k}\rbrace_{0\leq i\leq 2n_{x}}$ and $\lbrace\mathbf{q}^{*}_{i,k+1|k}\rbrace_{0\leq i\leq 2n_{x}}$ are the sigma points, respectively, generated and propagated to predict observation $\mathbf{y}_{k+1}$ as $\hat{\mathbf{y}}_{k+1|k}$ in the measurement update. The adversary's forward UKF recursions are \cite{simon2006optimal}
\par\noindent\small
\begin{align}
    &\textrm{Time update:}\;\;\;\lbrace\mathbf{s}_{i,k}\rbrace_{0\leq i\leq 2n_{x}}=S_{gen}(\hat{\mathbf{x}}_{k},\bm{\Sigma}_{k}),\label{eqn:forward ukf prediction sigma points}\\
    &\mathbf{s}^{*}_{i,k+1|k}=f(\mathbf{s}_{i,k})\;\;\;\forall i=0,1,\hdots,2n_{x},\nonumber\\
    &\hat{\mathbf{x}}_{k+1|k}=\sum_{i=0}^{2n_{x}}\omega_{i}\mathbf{s}^{*}_{i,k+1|k},\label{eqn:forward ukf x predict}\\
    &\bm{\Sigma}_{k+1|k}=\sum_{i=0}^{2n_{x}}\omega_{i}\mathbf{s}^{*}_{i,k+1|k}(\mathbf{s}^{*}_{i,k+1|k})^{T}-\hat{\mathbf{x}}_{k+1|k}(\hat{\mathbf{x}}_{k+1|k})^{T}+\mathbf{Q},\nonumber\\
    &\textrm{Measurement update:}\nonumber\\
    &\lbrace\mathbf{q}_{i,k+1|k}\rbrace_{0\leq i\leq 2n_{x}}=S_{gen}(\hat{\mathbf{x}}_{k+1|k},\bm{\Sigma}_{k+1|k}),\label{eqn:forward UKF update sigma points}\\
    &\mathbf{q}^{*}_{i,k+1|k}=h(\mathbf{q}_{i,k+1|k})\;\;\;\forall i=0,1,\hdots,2n_{x},\nonumber\\
     &\hat{\mathbf{y}}_{k+1|k}=\sum_{i=0}^{2n_{x}}\omega_{i}\mathbf{q}^{*}_{i,k+1|k},\label{eqn:forward ukf y predict}\\
    &\bm{\Sigma}^{y}_{k+1}=\sum_{i=0}^{2n_{x}}\omega_{i}\mathbf{q}^{*}_{i,k+1|k}(\mathbf{q}^{*}_{i,k+1|k})^{T}-\hat{\mathbf{y}}_{k+1|k}(\hat{\mathbf{y}}_{k+1|k})^{T}+\mathbf{R},\nonumber\\
    &\bm{\Sigma}^{xy}_{k+1}=\sum_{i=0}^{2n_{x}}\omega_{i}\mathbf{q}_{i,k+1|k}(\mathbf{q}^{*}_{i,k+1|k})^{T}-\hat{\mathbf{x}}_{k+1|k}(\hat{\mathbf{y}}_{k+1|k})^{T},\nonumber\\
    &\mathbf{K}_{k+1}=\bm{\Sigma}^{xy}_{k+1}\left(\bm{\Sigma}^{y}_{k+1}\right)^{-1},\nonumber\\
    &\hat{\mathbf{x}}_{k+1}=\hat{\mathbf{x}}_{k+1|k}+\mathbf{K}_{k+1}(\mathbf{y}_{k+1}-\hat{\mathbf{y}}_{k+1|k}),\label{eqn:forward ukf x update}\\
    &\bm{\Sigma}_{k+1}=\bm{\Sigma}_{k+1|k}-\mathbf{K}_{k+1}\bm{\Sigma}^{y}_{k+1}\mathbf{K}_{k+1}^{T}.\label{eqn:forward UKF sigma update}
\end{align}
\normalsize

\noindent\textbf{Inverse UKF:} The IUKF seeks to find an estimate $\Hat{\Hat{\mathbf{x}}}_{k}$ of the state estimate $\hat{\mathbf{x}}_{k}$ given our actual states $\lbrace\mathbf{x}_{j}\rbrace_{1\leq j\leq k}$ and observations $\lbrace\mathbf{a}_{j}\rbrace_{1\leq j\leq k}$ given by \eqref{eqn:observation a}. To this end, the inverse filter has $\hat{\mathbf{x}}_{k}$ as the state and $\mathbf{a}_{k}$ as the observation. Substituting \eqref{eqn:observation y}, \eqref{eqn:forward ukf x predict} and \eqref{eqn:forward ukf y predict} in \eqref{eqn:forward ukf x update}, yields the state transition
\par\noindent\small
\begin{align}
    \hat{\mathbf{x}}_{k+1}&=\sum_{i=0}^{2n_{x}}\omega_{i}\left(\mathbf{s}^{*}_{i,k+1|k}-\mathbf{K}_{k+1}\mathbf{q}^{*}_{i,k+1|k}\right)+\mathbf{K}_{k+1}h(\mathbf{x}_{k+1})\nonumber\\
    &\;\;+\mathbf{K}_{k+1}\mathbf{v}_{k+1}.\label{eqn: IUKF state transition detail}
\end{align}
\normalsize
In this state transition, $\mathbf{x}_{k+1}$ is a known exogenous input while $\mathbf{v}_{k+1}$ represents the process noise involved. The propagated sigma points $\lbrace\mathbf{s}^{*}_{i,k+1|k}\rbrace_{0\leq i\leq 2n_{x}}$ and $\lbrace\mathbf{q}^{*}_{i,k+1|k}\rbrace_{0\leq i\leq 2n_{x}}$, and the gain matrix $\mathbf{K}_{k+1}$ are functions of the first set of sigma points $\lbrace\mathbf{s}_{i,k}\rbrace_{0\leq i\leq 2n_{x}}$. These sigma points in turn, are obtained deterministically from the previous state estimate $\hat{\mathbf{x}}_{k}$ and covariance matrix $\bm{\Sigma}_{k}$ using \eqref{eqn:forward ukf prediction sigma points}. Hence, IUKF's state transition, \emph{under the assumption that parameter $\kappa$ is known to the defender}, is
\par\noindent\small
\begin{align}
    \hat{\mathbf{x}}_{k+1}=\widetilde{f}(\hat{\mathbf{x}}_{k},\bm{\Sigma}_{k},\mathbf{x}_{k+1},\mathbf{v}_{k+1}).\label{eqn:inverse ukf state transition}
\end{align}
\normalsize
Note that the process noise $\mathbf{v}_{k+1}$ is non-additive because $\mathbf{K}_{k+1}$ depends on the previous estimates. Furthermore, the covariance matrix $\bm{\Sigma}_{k}$ does not depend on the current forward filter's observation $\mathbf{y}_{k}$. It is evaluated recursively using the previous estimates and initial covariance estimate $\bm{\Sigma}_{0}$. The inverse filter treats $\bm{\Sigma}_{k}$ as a known exogenous input in the state transition \eqref{eqn:inverse ukf state transition}. It approximates $\bm{\Sigma}_{k}$ as $\bm{\Sigma}_{k}^{*}$ by computing the covariance matrix using its own previous estimate, i.e. $\Hat{\Hat{\mathbf{x}}}_{k}$, in the same recursive manner as the forward filter estimates at its estimate $\hat{\mathbf{x}}_{k}$ (using \eqref{eqn:forward UKF sigma update} after computing gain matrix $\mathbf{K}_{k+1}$ from the generated sigma points).

The presence of non-additive noise term in the state transition \eqref{eqn:inverse ukf state transition} leads us to consider an augmented state vector $\mathbf{z}_{k}=[\hat{\mathbf{x}}_{k}^{T},\mathbf{v}_{k+1}^{T}]^{T}$ of dimension $n_{z}=n_{x}+n_{y}$ for the inverse filter formulation. Here, \eqref{eqn:inverse ukf state transition} is expressed in terms of $\mathbf{z}_{k}$ as $\hat{\mathbf{x}}_{k+1}=\widetilde{f}(\mathbf{z}_{k},\bm{\Sigma}_{k},\mathbf{x}_{k+1})$. Denote $\hat{\mathbf{z}}_{k}=[\Hat{\Hat{\mathbf{x}}}_{k}^{T},\mathbf{0}_{1\times n_{y}}]^T$ and $\overline{\bm{\Sigma}}^{z}_{k}=\left[\begin{smallmatrix}\overline{\bm{\Sigma}}_{k}&\mathbf{0}_{n_{x}\times n_{y}}\\\mathbf{0}_{n_{y}\times n_{x}}&\mathbf{R}\end{smallmatrix}\right]$. Considering $\overline{\kappa}$ as IUKF's scaling parameter, the sigma points $\{\overline{\mathbf{s}}_{j,k}\}_{0\leq j\leq 2n_{z}}$ are generated from $\hat{\mathbf{z}}_{k}$ and $\overline{\bm{\Sigma}}^{z}_{k}$ similar to \eqref{eqn:sigma points generation} with weights $\overline{\omega}_{j}$. Finally, the IUKF's recursions to infer the estimate $\Hat{\Hat{\mathbf{x}}}_{k}$ and the associated error covariance estimate $\overline{\bm{\Sigma}}_{k}$ are
\par\noindent\small
\begin{align}
    &\textrm{Time update:}\nonumber\\
    &\overline{s}^{*}_{j,k+1|k}=\widetilde{f}(\overline{\mathbf{s}}_{j,k},\bm{\Sigma}_{k}^{*},\mathbf{x}_{k+1})\;\;\;\forall j=0,1,\hdots,2n_{z},\\
    &\Hat{\Hat{\mathbf{x}}}_{k+1|k}=\sum_{j=0}^{2n_{z}}\overline{\omega}_{j}\overline{\mathbf{s}}^{*}_{j,k+1|k},\label{eqn:IUKF state predict}\\
    &\overline{\bm{\Sigma}}_{k+1|k}=\sum_{j=0}^{2n_{z}}\overline{\omega}_{j}\overline{\mathbf{s}}^{*}_{j,k+1|k}(\overline{\mathbf{s}}^{*}_{j,k+1|k})^{T}-\Hat{\Hat{\mathbf{x}}}_{k+1|k}\Hat{\Hat{\mathbf{x}}}_{k+1|k}^{T},\nonumber\\
    &\textrm{Measurement update:}\;\;\;\mathbf{a}^{*}_{j,k+1|k}=g(\overline{\mathbf{s}}^{*}_{j,k+1|k})\;\;\;\forall j=0,1,\hdots,2n_{z},\nonumber\\
    &\hat{\mathbf{a}}_{k+1|k}=\sum_{j=0}^{2n_{z}}\overline{\omega}_{j}\mathbf{a}^{*}_{j,k+1|k},\nonumber\\
    &\overline{\bm{\Sigma}}^{a}_{k+1}=\sum_{j=0}^{2n_{z}}\overline{\omega}_{j}\mathbf{a}^{*}_{j,k+1|k}(\mathbf{a}^{*}_{j,k+1|k})^{T}-\hat{\mathbf{a}}_{k+1|k}\hat{\mathbf{a}}_{k+1|k}^{T}+\bm{\Sigma}_{\epsilon},\nonumber
\end{align}
\begin{align}
    &\textrm{Measurement update (contd.):}\nonumber\\
    &\overline{\bm{\Sigma}}^{xa}_{k+1}=\sum_{j=0}^{2n_{z}}\overline{\omega}_{j}\overline{\mathbf{s}}^{*}_{j,k+1|k}(\mathbf{a}^{*}_{j,k+1|k})^{T}-\Hat{\Hat{\mathbf{x}}}_{k+1|k}\hat{\mathbf{a}}_{k+1|k}^{T},\nonumber\\
    &\overline{\mathbf{K}}_{k+1}=\overline{\bm{\Sigma}}^{xa}_{k+1}\left(\overline{\bm{\Sigma}}^{a}_{k+1}\right)^{-1}\nonumber\\
    &\Hat{\Hat{x}}_{k+1}=\Hat{\Hat{x}}_{k+1|k}+\overline{\mathbf{K}}_{k+1}(\mathbf{a}_{k+1}-\hat{\mathbf{a}}_{k+1|k}),\nonumber\\
    &\overline{\bm{\Sigma}}_{k+1}=\overline{\bm{\Sigma}}_{k+1|k}-\overline{\mathbf{K}}_{k+1}\overline{\bm{\Sigma}}^{a}_{k+1}\overline{\mathbf{K}}_{k}^{T}.\nonumber
\end{align}
\normalsize
\begin{remark}
The IUKF recursions follow from the augmented state formulation of UKF for non-additive noise \cite{wan2000unscented} so that sigma points generated are in higher dimensional state space ($n_{z}=n_{x}+n_{y}$ dimensional) as compared to those generated in forward UKF ($n_{x}$ dimensional). Whereas forward UKF requires a new set for the measurement update, the IUKF generates these points only once for the time update considering the state transition's process noise statistics.
\end{remark}
\begin{remark}
Unlike IKF \cite{krishnamurthy2019how} and IEKF \cite{singh2022inverse}, the forward gain matrix of IUKF $\mathbf{K}_{k+1}$ is not treated as a time-varying parameter of \eqref{eqn: IUKF state transition detail}. In KF, the gain matrix is deterministic and fully determined by the model parameters given the initial covariance estimate $\bm{\Sigma}_{0}$. Similarly, in EKF, the gain matrix, computed from the covariance estimates, depends on the linearized model functions at the state estimate. But these covariance estimates in UKF are obtained as the weighted average of the generated sigma points, which are explicit functions of the state estimates ($\hat{\mathbf{x}}_{k}$ and $\hat{\mathbf{x}}_{k+1|k}$), prohibiting the inverse filter from treating $\mathbf{K}_{k+1}$ as a parameter of \eqref{eqn: IUKF state transition detail}.
\end{remark}
\section{Stability Guarantees}
\label{sec:stability}
We show that the IUKF's error dynamics satisfy the stability conditions of a general UKF under mild system conditions if the forward UKF is stable. Therefore, we first examine the stochastic stability of forward UKF. We consider the general case of time-varying process and measurement noise covariances $\mathbf{Q}_{k}$, $\mathbf{R}_{k}$ and $\overline{\mathbf{R}}_{k}$ instead of $\mathbf{Q}$, $\mathbf{R}$ and $\bm{\Sigma}_{\epsilon}$, respectively. Recall the definition of exponential-mean-squared-boundedness of a stochastic process.
\begin{definition}[Exponential mean-squared boundedness] \cite{reif1999stochastic} The stochastic process $\{\bm{\zeta}_{k} \}_{k \geq 0}$ is said to be exponentially bounded in the mean-squared sense if there exist real numbers $\eta>0$, $\nu>0$ and $0<\lambda<1$ such that $\mathbb{E}\left[\|\bm{\zeta}_{k}\|_{2}^{2}\right]\leq \eta\mathbb{E}\left[\|\bm{\zeta}_{0}\|_{2}^{2}\right]\lambda^{k}+\nu$ holds for every $k\geq 0$.
\end{definition}

\noindent\textbf{Forward UKF:} Denote the forward UKF's state prediction, state estimation, and measurement prediction errors by $\widetilde{\mathbf{x}}_{k+1|k}\doteq\mathbf{x}_{k+1}-\hat{\mathbf{x}}_{k+1|k}$, $\widetilde{\mathbf{x}}_{k}\doteq\mathbf{x}_{k}-\hat{\mathbf{x}}_{k}$ and $\widetilde{\mathbf{y}}_{k+1}\doteq\mathbf{y}_{k+1}-\hat{\mathbf{y}}_{k+1|k}$, respectively. From \eqref{eqn:state evolution x} and \eqref{eqn:forward ukf x predict}, we have $\widetilde{\mathbf{x}}_{k+1|k}=f(\mathbf{x}_{k})+\mathbf{w}_{k}-\sum_{i=0}^{2n_{x}}\omega_{i}\mathbf{s}^{*}_{i,k+1|k}$, which on substituting $\mathbf{s}^{*}_{i,k+1|k}=f(\mathbf{s}_{i,k})$ yields 
  $  \widetilde{\mathbf{x}}_{k+1|k}=f(\mathbf{x}_{k})+\mathbf{w}_{k}-\sum_{i=0}^{2n_{x}}\omega_{i}f(\mathbf{s}_{i,k})$. 
Using the first-order Taylor series expansion of $f(\cdot)$ at $\hat{\mathbf{x}}_{k}$, 
we have $\widetilde{\mathbf{x}}_{k+1|k}\approx f(\hat{\mathbf{x}}_{k})+\mathbf{F}_{k}(\mathbf{x}_{k}-\hat{\mathbf{x}}_{k})+\mathbf{w}_{k}-\sum_{i=0}^{2n_{x}}\omega_{i}(f(\hat{\mathbf{x}}_{k})+\mathbf{F}_{k}(\mathbf{s}_{i,k}-\hat{\mathbf{x}}_{k}))$, where $\mathbf{F}_{k}\doteq\frac{\partial f(\mathbf{x})}{\partial\mathbf{x}}\vert_{\mathbf{x}=\hat{\mathbf{x}}_{k}}$. The sigma points $\lbrace\mathbf{s}_{i,k}\rbrace_{0\leq i\leq 2n_{x}}$ are chosen symmetrically about $\hat{\mathbf{x}}_{k}$. Substituting for $\mathbf{s}_{i,k}$ in terms of $\hat{\mathbf{x}}_{k}$ and $\bm{\Sigma}_{k}$ using \eqref{eqn:forward ukf prediction sigma points} simplifies the state prediction error to 
    $\widetilde{\mathbf{x}}_{k+1|k}\approx\mathbf{F}_{k}\widetilde{\mathbf{x}}_{k}+\mathbf{w}_{k}$.
Similar to \cite{xiong2006performance_ukf, li2012stochastic_ukf}, we introduce an unknown instrumental diagonal matrix $\mathbf{U}^{x}_{k}\in\mathbb{R}^{n_{x}\times n_{x}}$ to account for the linearization errors and obtain 
\par\noindent\small
\begin{align}
    \widetilde{\mathbf{x}}_{k+1|k}=\mathbf{U}^{x}_{k}\mathbf{F}_{k}\widetilde{\mathbf{x}}_{k}+\mathbf{w}_{k}.\label{eqn:linearized x}
\end{align}
\normalsize
Linearizing $h(\cdot)$ in \eqref{eqn:observation y} and introducing unknown diagonal matrix $\mathbf{U}^{y}_{k}\in\mathbb{R}^{n_{y}\times n_{y}}$ in \eqref{eqn:forward ukf y predict} yields
\par\noindent\small
\begin{align}
    \widetilde{\mathbf{y}}_{k+1}=\mathbf{U}^{y}_{k+1}\mathbf{H}_{k+1}\widetilde{\mathbf{x}}_{k+1|k}+\mathbf{v}_{k+1},\label{eqn:linearized y}
\end{align}
\normalsize
where $\mathbf{H}_{k+1}\doteq\frac{\partial h(\mathbf{x})}{\partial\mathbf{x}}\vert_{\mathbf{x}=\hat{\mathbf{x}}_{k+1|k}}$. Using \eqref{eqn:forward ukf x update}, we have $\widetilde{\mathbf{x}}_{k}=\widetilde{\mathbf{x}}_{k|k-1}-\mathbf{K}_{k}\widetilde{\mathbf{y}}_{k}$, which when substituted in \eqref{eqn:linearized x} with \eqref{eqn:linearized y} yields the forward UKF's prediction error dynamics as
\par\noindent\small
\begin{align}
    \widetilde{\mathbf{x}}_{k+1|k}=\mathbf{U}^{x}_{k}\mathbf{F}_{k}(\mathbf{I}-\mathbf{K}_{k}\mathbf{U}^{y}_{k}\mathbf{H}_{k})\widetilde{\mathbf{x}}_{k|k-1}-\mathbf{U}^{x}_{k}\mathbf{F}_{k}\mathbf{K}_{k}\mathbf{v}_{k}+\mathbf{w}_{k}.\label{eqn:forward ukf error dynamics}
\end{align}
\normalsize
Denote the true prediction covariance by  $\mathbf{P}_{k+1|k}=\mathbb{E}\left[\widetilde{\mathbf{x}}_{k+1|k}\widetilde{\mathbf{x}}_{k+1|k}^{T}\right]$. Define $\delta\mathbf{P}_{k+1|k}$ as the difference of estimated prediction covariance $\bm{\Sigma}_{k+1|k}$ and the true prediction covariance $\mathbf{P}_{k+1|k}$, while $\Delta\mathbf{P}_{k+1|k}$ is the error in the approximation of the expectation\\
\small$\mathbb{E}\left[\mathbf{U}^{x}_{k}\mathbf{F}_{k}(\mathbf{I}-\mathbf{K}_{k}\mathbf{U}^{y}_{k}\mathbf{H}_{k})\widetilde{\mathbf{x}}_{k|k-1}\widetilde{\mathbf{x}}_{k|k-1}^{T}(\mathbf{I}-\mathbf{K}_{k}\mathbf{U}^{y}_{k}\mathbf{H}_{k})^{T}\mathbf{F}_{k}^{T}\mathbf{U}^{x}_{k}\right]$
\normalsize
by $\mathbf{U}^{x}_{k}\mathbf{F}_{k}(\mathbf{I}-\mathbf{K}_{k}\mathbf{U}^{y}_{k}\mathbf{H}_{k})\bm{\Sigma}_{k|k-1}(\mathbf{I}-\mathbf{K}_{k}\mathbf{U}^{y}_{k}\mathbf{H}_{k})^{T}\mathbf{F}_{k}^{T}\mathbf{U}^{x}_{k}$. Denoting $\hat{\mathbf{Q}}_{k}=\mathbf{Q}_{k}+\mathbf{U}^{x}_{k}\mathbf{F}_{k}\mathbf{K}_{k}\mathbf{R}_{k}\mathbf{K}_{k}^{T}\mathbf{F}_{k}^{T}\mathbf{U}^{x}_{k}+\delta\mathbf{P}_{k+1|k}+\Delta\mathbf{P}_{k+1|k}$ and using \eqref{eqn:forward ukf error dynamics} similar to \cite{xiong2006performance_ukf, li2012stochastic_ukf}, we have
\par\noindent\small
\begin{align}
    &\bm{\Sigma}_{k+1|k}=\hat{\mathbf{Q}}_{k}\nonumber\\
    &+\mathbf{U}^{x}_{k}\mathbf{F}_{k}(\mathbf{I}-\mathbf{K}_{k}\mathbf{U}^{y}_{k}\mathbf{H}_{k})\bm{\Sigma}_{k|k-1}(\mathbf{I}-\mathbf{K}_{k}\mathbf{U}^{y}_{k}\mathbf{H}_{k})^{T}\mathbf{F}_{k}^{T}\mathbf{U}^{x}_{k}.\label{eqn:forward UKF P predict}
\end{align}
\normalsize
Similarly, we have
\par\noindent\small
\begin{align}
    \bm{\Sigma}^{y}_{k+1}&=\mathbf{U}^{y}_{k+1}\mathbf{H}_{k+1}\bm{\Sigma}_{k+1|k}\mathbf{H}_{k+1}^{T}\mathbf{U}^{y}_{k+1}+\hat{\mathbf{R}}_{k+1},\label{eqn:forward UKF P y}\\
    \bm{\Sigma}^{xy}_{k+1}&=\begin{cases}\bm{\Sigma}_{k+1|k}\mathbf{U}^{xy}_{k+1}\mathbf{H}_{k+1}^{T}\mathbf{U}^{y}_{k+1}, & n_{x}\geq n_{y}\\
    \bm{\Sigma}_{k+1|k}\mathbf{H}_{k+1}^{T}\mathbf{U}^{y}_{k+1}\mathbf{U}^{xy}_{k+1}, & n_{x}<n_{y}\end{cases},\label{eqn:forward UKF P xy}
\end{align}
\normalsize
where $\hat{\mathbf{R}}_{k+1}=\mathbf{R}_{k+1}+\Delta\mathbf{P}^{y}_{k+1}+\delta\mathbf{P}^{y}_{k+1}$ with $\delta\mathbf{P}^{y}_{k+1}$ and $\Delta\mathbf{P}^{y}_{k+1}$, respectively, accounting for the difference in true and estimated measurement prediction covariances, and error in the approximation of the expectation. Also, $\mathbf{U}^{xy}_{k+1}$ is an unknown instrumental matrix introduced to account for errors in the estimated cross-covariance $\bm{\Sigma}^{xy}_{k+1}$. Following Theorem~\ref{theorem:forward ukf stability} provides sufficient conditions for UKF's stability.
\begin{theorem}[Stochastic stability of forward UKF]
\label{theorem:forward ukf stability}
Consider the non-linear stochastic system given by \eqref{eqn:state evolution x} and \eqref{eqn:observation y} and the forward UKF. Let the following hold true.\\
\textbf{C1.} There exist positive real numbers $\bar{f}$, $\bar{h}$, $\bar{\alpha}$, $\bar{\beta}$, $\bar{\gamma}$, $\underline{\sigma}$, $\bar{\sigma}$, $\bar{q}$, $\bar{r}$, $\hat{q}$ and $\hat{r}$ such that the following bounds are fulfilled for all $k\geq 0$.
\par\noindent\small
    \begin{align*}        &\|\mathbf{F}_{k}\|\leq\bar{f},\;\;\|\mathbf{H}_{k}\|\leq\bar{h},\;\;\|\mathbf{U}^{x}_{k}\|\leq\bar{\alpha},\;\;\|\mathbf{U}^{y}_{k}\|\leq\bar{\beta},\;\;\|\mathbf{U}^{xy}_{k}\|\leq\bar{\gamma},\\        &\mathbf{Q}_{k}\preceq\bar{q}\mathbf{I},\;\;\mathbf{R}_{k}\preceq\bar{r}\mathbf{I},\;\;\hat{q}\mathbf{I}\preceq\hat{\mathbf{Q}}_{k},\;\;\hat{r}\mathbf{I}\preceq\hat{\mathbf{R}}_{k},\;\;\underline{\sigma}\mathbf{I}\preceq\bm{\Sigma}_{k|k-1}\preceq\bar{\sigma}\mathbf{I}.
    \end{align*}
\normalsize
\textbf{C2.} $\mathbf{U}^{x}_{k}$ and $\mathbf{F}_{k}$ are non-singular for every $k\geq 0$.\\
\textbf{C3.} The constants satisfy the inequality $\bar{\sigma}\bar{\gamma}\bar{h}^{2}\bar{\beta}^{2}<\hat{r}$.\\
Then, the prediction error $\widetilde{\mathbf{x}}_{k|k-1}$ and hence, the estimation error $\widetilde{\mathbf{x}}_{k}$ of the forward UKF are exponentially bounded in mean-squared sense and bounded with probability one.
\end{theorem}
\begin{proof}
We omit the proof because of the paucity of space. We refer the reader to Appendix I of the online version of this paper uploaded at \url{bit.ly/3z16Hst}.
\end{proof}
\begin{remark}
The UKF stability is discussed in \cite[Theorem~1]{xiong2006performance_ukf} for only linear measurements while our case is non-linear. 
Interestingly, \cite[Theorem~1]{xiong2006performance_ukf} requires a lower bound on measurement noise covariance $\mathbf{R}_{k}$ and also, an upper bound on $\hat{\mathbf{Q}}_{k}$. But $\hat{\mathbf{Q}}_{k}$ is not upper bounded in Theorem~\ref{theorem:forward ukf stability}. We require upper (lower) bounds on noise covariances $\mathbf{Q}_{k}$ ($\hat{\mathbf{Q}}_{k}$) and $\mathbf{R}_{k}$ ($\hat{\mathbf{R}}_{k}$). 
Both $\hat{\mathbf{Q}}_{k}$ and $\hat{\mathbf{R}}_{k}$ may be made positive definite to satisfy the lower bounds by enlarging the noise covariance matrices $\mathbf{Q}_{k}$ and $\mathbf{R}_{k}$, respectively. 
This enhances the stability of the filter \cite{xiong2006performance_ukf,xiong2007authorreply}.
\end{remark}
\noindent\textbf{Inverse UKF:} Similar to the forward UKF, we introduce unknown matrices $\overline{\mathbf{U}}^{x}_{k}$ and $\overline{\mathbf{U}}^{a}_{k}$ to account for the errors in linearization of functions $\widetilde{f}(\cdot)$ and $g(\cdot)$, respectively, and $\overline{\mathbf{U}}^{xa}_{k}$ for the errors in cross-covariance matrix estimation. Also, $\hat{\overline{\mathbf{Q}}}_{k}$ and $\hat{\overline{\mathbf{R}}}_{k}$ denote the counterparts of $\hat{\mathbf{Q}}_{k}$ and $\hat{\mathbf{R}}_{k}$, respectively, in the IUKF's error dynamics. Define $\widetilde{\mathbf{F}}_{k}\doteq\left.\frac{\partial\widetilde{f}(\mathbf{x},\bm{\Sigma}_{k},\mathbf{x}_{k+1},\mathbf{0})}{\partial\mathbf{x}}\right\vert_{\mathbf{x}=\Hat{\Hat{\mathbf{x}}}_{k}}$ and $\mathbf{G}_{k}\doteq\left.\frac{\partial g(\mathbf{x})}{\partial\mathbf{x}}\right\vert_{\mathbf{x}=\Hat{\Hat{\mathbf{x}}}_{k|k-1}}$. In approximating $\bm{\Sigma}_{k}$ by $\bm{\Sigma}^{*}_{k}$ in IUKF, we ignore any errors. Assume that the forward gain $\mathbf{K}_{k+1}$ computed from $\Hat{\Hat{\mathbf{x}}}_{k}$ is approximately same as that computed from $\hat{\mathbf{x}}_{k}$ in forward UKF. Additionally, these approximation errors are bounded by positive constants because 
$\bm{\Sigma}_{k}$ and 
$\mathbf{K}_{k+1}$ can be proved to be bounded matrices under the IUKF's stability assumptions. The bounds required on various matrices for forward UKF's stability are also satisfied when these matrices are evaluated by IUKF at its own estimates, i.e., $\left\|\frac{\partial f(\mathbf{x})}{\partial\mathbf{x}}\right\|\leq \bar{f}$ and $\left\|\frac{\partial h(\mathbf{x})}{\partial\mathbf{x}}\right\|\leq \bar{h}$ where $\mathbf{x}$ is any sigma-point of the forward or inverse UKF.
We state the stability conditions of IUKF in the following theorem.
\begin{theorem}[Stochastic stability of IUKF]
\label{theorem:IUKF stability}
Consider the adversary's forward UKF that is stable as per Theorem~\ref{theorem:forward ukf stability}. Additionally, assume that the following hold true.\\
\textbf{C4.} There exist positive real numbers $\bar{g},\bar{c},\bar{d},\bar{\epsilon},\hat{c},\hat{d},\underline{p}$ and $\bar{p}$ such that the following bounds are fulfilled for all $k\geq 0$.
\par\noindent\small
\begin{align*}      &|\mathbf{G}_{k}\|\leq\bar{g},\;\;\|\overline{\mathbf{U}}^{a}_{k}\|\leq\bar{c},\;\;\|\overline{\mathbf{U}}^{xa}_{k}\|\leq\bar{d},\;\;\overline{\mathbf{R}}_{k}\preceq\bar{\epsilon}\mathbf{I},\;\;\hat{c}\mathbf{I}\preceq\hat{\overline{\mathbf{Q}}}_{k},\\
&\hat{d}\mathbf{I}\preceq\hat{\overline{R}}_{k},\;\;\underline{p}\mathbf{I}\preceq\overline{\bm{\Sigma}}_{k|k-1}\preceq\bar{p}\mathbf{I}.
\end{align*}
\normalsize
\textbf{C5.} There exist a real constant $\underline{y}$ (not necessarily positive) such that $\bm{\Sigma}^{y}_{k}\succeq\underline{y}\mathbf{I}$ for all $k\geq 0$.\\
\textbf{C6.} The functions $f(\cdot)$ and $h(\cdot)$ have bounded outputs i.e. $\|f(\cdot)\|_{2}\leq\delta_{f}$ and $\|h(\cdot)\|_{2}\leq\delta_{h}$ for some real positive numbers $\delta_{f}$ and $\delta_{h}$.\\
\textbf{C7.} For all $k\geq 0$, $\widetilde{\mathbf{F}}_{k}$ is non-singular and its inverse satisfies $\|\widetilde{\mathbf{F}}^{-1}_{k}\|\leq\bar{a}$ for some positive real constant $\bar{a}$.\\
Then, the IUKF's state estimation error is exponentially bounded in mean-squared sense and bounded with probability one provided that the constants satisfy the inequality $\bar{p}\bar{d}\bar{g}^{2}\bar{c}^{2}<\hat{d}$.
\end{theorem}
\begin{proof}
We refer the reader to Appendix II of the online version of this paper uploaded at \url{bit.ly/3z16Hst}.
\end{proof}

Note that, while there are no constraints on the constant $\underline{y}$ in \textbf{C5}, Theorem~\ref{theorem:IUKF stability} requires an additional lower bound on $\bm{\Sigma}^{y}_{k}$ which was not needed for forward UKF's stability. Also, $\underline{y}\neq 0$ because $(\bm{\Sigma}^{y}_{k})^{-1}$ exists for forward UKF to compute its gain. \textbf{C5} and \textbf{C6} are necessary to upper-bound the Jacobian $\widetilde{\mathbf{F}}_{k}$. The bounds $\|f(\cdot)\|_{2}\leq\delta_{f}$ and $\|h(\cdot)\|_{2}\leq\delta_{h}$ help to bound the magnitude of the propagated sigma points $\lbrace\mathbf{s}^{*}_{i,k+1|k}\rbrace$ and $\lbrace\mathbf{q}^{*}_{i,k+1|k}\rbrace$, respectively, generated from a state estimate, which in turn, upper bounds the various covariance estimates computed from them. Also, the computation of gain matrix $\mathbf{K}_{k}$ involves $(\bm{\Sigma}^{y}_{k})^{-1}$, which is upper-bounded under the lower bound assumed in \textbf{C5}.

\section{Numerical Experiments}
\label{sec:numericals}
We demonstrate the proposed IUKF's performance considering different example systems and comparing the state estimation error with the corresponding RCRLB. CRLB is a widely used performance measure for an estimator providing a lower bound on mean-squared error (MSE). For discrete-time non-linear filtering, we employ RCRLB. Denote the state vector series by $X^{k}=\lbrace\mathbf{x}_{0},\mathbf{x}_{1},\hdots,\mathbf{x}_{k}\rbrace$ while the noisy observations series by $Y^{k}=\lbrace\mathbf{y}_{0},\mathbf{y}_{1},\hdots,\mathbf{y}_{k}\rbrace$. The joint probability density of pair $(Y^{k},X^{k})$ is $p(Y^{k},X^{k})$. 

The RCRLB at $k$-th time instant for the estimate $\hat{\mathbf{x}}_{k}$ (a function of $Y^{k}$) of the true state $\mathbf{x}_{k}$ is defined as
  $\mathbb{E}\left[(\mathbf{x}_{k}-\hat{\mathbf{x}}_{k})(\mathbf{x}_{k}-\hat{\mathbf{x}}_{k})^{T}\right]\succeq\mathbf{J}_{k}^{-1}$,
where 
    $\mathbf{J}_{k}=\mathbb{E}\left[-\frac{\partial^{2}\ln{p(Y^{k},X^{k})}}{\partial\mathbf{x}_{k}^{2}}\right]$, 
is the Fisher information matrix \cite{tichavsky1998posterior}. Here, $\frac{\partial^{2}(\cdot)}{\partial\mathbf{x}^{2}}$ denotes the Hessian comprising of the second order partial derivatives. The sequence $\lbrace\mathbf{J}_{k}\rbrace$ are computed recursively as 
    $\mathbf{J}_{k}=\mathbf{D}_{k}^{22}-\mathbf{D}_{k}^{21}(\mathbf{J}_{k-1}+\mathbf{D}_{k}^{11})^{-1}\mathbf{D}_{k}^{12}$, 
where \small
    $\mathbf{D}_{k}^{11}=\mathbb{E}\left[-\frac{\partial^{2}\ln{p(\mathbf{x}_{k}\vert\mathbf{x}_{k-1})}}{\partial\mathbf{x}_{k-1}^{2}}\right]$, 
    $\mathbf{D}_{k}^{12}=\mathbb{E}\left[-\frac{\partial^{2}\ln{p(\mathbf{x}_{k}\vert\mathbf{x}_{k-1})}}{\partial\mathbf{x}_{k}\partial\mathbf{x}_{k-1}}\right]=(\mathbf{D}_{k}^{21})^{T}$, and 
    $\mathbf{D}_{k}^{22}=\mathbb{E}\left[-\frac{\partial^{2}\ln{p(\mathbf{x}_{k}\vert\mathbf{x}_{k-1})}}{\partial\mathbf{x}_{k}^{2}}\right]+\mathbb{E}\left[-\frac{\partial^{2}\ln{p(\mathbf{y}_{k}\vert\mathbf{x}_{k})}}{\partial\mathbf{x}_{k}^{2}}\right]$\cite{tichavsky1998posterior}.
\normalsize
For the non-linear system \eqref{eqn:state evolution x} and \eqref{eqn:observation y} with Gaussian noises, the $\lbrace\mathbf{J}_{k}\rbrace$ recursions simplifies to 
    $\mathbf{J}_{k}=\mathbf{Q}^{-1}+\mathbf{H}_{k}^{T}\mathbf{R}^{-1}\mathbf{H}_{k}-\mathbf{Q}^{-1}\mathbf{F}_{k}(\mathbf{J}_{k-1}+\mathbf{F}_{k}^{T}\mathbf{Q}^{-1}\mathbf{F}_{k})^{-1}\mathbf{F}_{k}^{T}\mathbf{Q}^{-1}$, 
where $\mathbf{F}_{k}\doteq\frac{\partial f(\mathbf{x})}{\partial\mathbf{x}}|_{\mathbf{x}=\hat{\mathbf{x}}_{k}}$ and $\mathbf{H}_{k}\doteq\frac{\partial h(\mathbf{x})}{\partial\mathbf{x}}|_{\mathbf{x}=\hat{\mathbf{x}}_{k}}$ \cite{xiong2006performance_ukf}. Similarly, we can compute the posterior information matrix $\overline{\mathbf{J}}_{k}$ for the inverse filter's estimate $\Hat{\Hat{\mathbf{x}}}_{k}$. 
Throughout all experiments, the initial information matrices $\mathbf{J}_{0}$ and $\overline{\mathbf{J}}_{0}$ for the forward and inverse filters were set to $\bm{\Sigma}_{0}^{-1}$ and $\overline{\bm{\Sigma}}_{0}^{-1}$, respectively. 
\noindent\textbf{FM demodulation:} 
To compare IUKF's performance with IEKF \cite{singh2022inverse}, we consider FM demodulator system \cite[Sec. 8.2]{anderson2012optimal} with system model 
\par\noindent\small
\begin{align*}
&\mathbf{x}_{k+1}\doteq\left[\begin{smallmatrix}\lambda_{k+1}\\\theta_{k+1}\end{smallmatrix}\right]=\left[\begin{smallmatrix}\exp{(-T/\beta)}&0\\-\beta \exp{(-T/\beta)}-1&1\end{smallmatrix}\right]\left[\begin{smallmatrix}\lambda_{k}\\\theta_{k}\end{smallmatrix}\right]+\left[\begin{smallmatrix}1\\-\beta\end{smallmatrix}\right]w_{k},\\
&\mathbf{y}_{k}=\sqrt{2}\left[\begin{smallmatrix}\sin{\theta_{k}}\\\cos{\theta_{k}}\end{smallmatrix}\right]+\mathbf{v}_{k},\;\;\;
a_{k}=\hat{\lambda}_{k}^{2}+\epsilon_{k},
\end{align*}
\normalsize
with $w_{k}\sim\mathcal{N}(0,0.01)$, $\mathbf{v}_{k}\sim\mathcal{N}(\mathbf{0},\mathbf{I}_{2})$, $\epsilon_{k}\sim\mathcal{N}(0,5)$, $T=2\pi/16$ and $\beta=100$. Here, $\hat{\lambda}_{k}$ is the forward filter's estimate of $\lambda_{k}$. The initial state $\mathbf{x}_{0}\doteq[\lambda_{0},\theta_{0}]^{T}$ and its estimate for the forward EKF and UKF were set randomly with $\lambda_{0}\sim\mathcal{N}(0,1)$ and $\theta_{0}\sim\mathcal{U}[-\pi,\pi]$. For the inverse filters, the initial state estimate was chosen as $\mathbf{x}_{0}$. The initial covariance estimates $\bm{\Sigma}_{0}$ and $\overline{\bm{\Sigma}}_{0}$ were set to $10\mathbf{I}_{2}$ and $5\mathbf{I}_{2}$ for the forward and inverse filters, respectively. Also, for forward and inverse UKF, $\kappa$ and $\overline{\kappa}$ both were set to $1$, but IUKF assumed the forward UKF's $\kappa$ to be $2$. 
\begin{figure}
  \centering
  \includegraphics[width = \columnwidth]{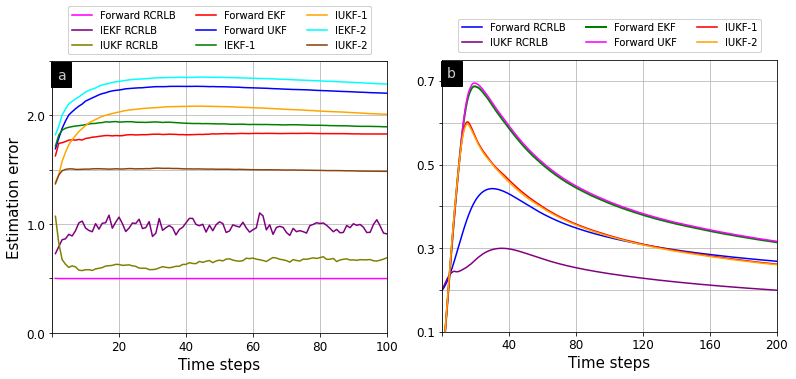}
  \caption{(a) Time-averaged RMSE and RCRLB 
  for FM demodulator system; (b) 
  As in (a) but for the vehicle reentry system.}
 \label{fig:UKF sim}
\end{figure}

Fig. \ref{fig:UKF sim}a shows the time-averaged root MSE (RMSE) and RCRLB for state estimation for forward and inverse EKF and UKF (IEKF-1 and IUKF-1), averaged over 500 runs. We also consider inverse filters, which assume a forward filter that is not the same as the true forward filter. In particular, IEKF-2 and IUKF-2, respectively, assume the forward filter to be EKF and UKF when the true forward filters are UKF and EKF, respectively. As compared to forward EKF, forward UKF has a higher estimation error for the FM demodulator system. Note that EKF and UKF's performance also depends on the system itself such that UKF may not necessarily provide better estimates. Although IUKF's error is higher than that of IEKF, IUKF estimates its state more accurately than forward UKF. On the contrary, IEKF's error is more than its corresponding forward filter. 

Interestingly, while IEKF's performance is adversely affected by the incorrect forward filter assumption (IEKF-2 case), IUKF-2 has a smaller estimation error than all other inverse filters. This suggests that the unscented transform approach is more robust to the (non-linear) inverse filter's dynamics than the linearization used in IEKF. We further conclude that assuming a forward UKF and using IUKF may provide estimation accuracy better than IEKF. Note that IUKF-2 outperforms IUKF-1 as well because the former's forward EKF is estimating the state more accurately than the latter's forward UKF. 
Furthermore, assuming forward UKF's $\kappa$ in IUKF to be different from its true value does not affect the estimation performance. The forward and inverse filters are compared only to highlight relative estimation accuracy. 

\noindent\textbf{Vehicle reentry:} 
Consider tracking reentry of a vehicle by a radar using range and bearing measurements. The vehicle reentry problem is used widely to illustrate UKF's performance \cite{julier1995new,sarkka2007unscented,lu2011two}. With $[\mathbf{x}_{k}]_{1}$ and $[\mathbf{x}_{k}]_{2}$ representing the position of the vehicle at $k$-th time instant, $[\mathbf{x}_{k}]_{3}$ and $[\mathbf{x}_{k}]_{4}$ denoting its velocity, and $[\mathbf{x}_{k}]_{5}$ its constant aerodynamic parameter, the continuous-time evolution of the vehicle's state follows $[\dot{\mathbf{x}}_{k}]_{1}=[\mathbf{x}_{k}]_{3}$, $[\dot{\mathbf{x}}_{k}]_{2}=[\mathbf{x}_{k}]_{4}$, $[\dot{\mathbf{x}}_{k}]_{3}=d_{k}[\mathbf{x}_{k}]_{3}+g_{k}[\mathbf{x}_{k}]_{1}+w_{1}$, $[\dot{\mathbf{x}}_{k}]_{4}=d_{k}[\mathbf{x}_{k}]_{4}+g_{k}[\mathbf{x}_{k}]_{2}+w_{2}$, $[\dot{\mathbf{x}}_{k}]_{5}=w_{3}$, where $[\dot{\mathbf{x}}_{k}]_{i}$ is the first-order partial derivative of $[\mathbf{x}_{k}]_{i}$ with respect to time, and $w_{1}, w_{2}$ and $w_{3}$ represent the zero-mean Gaussian process noise \cite{julier1995new}. We consider the discretized version of this system with a time step of $0.1$ sec in our experiment. The quantities $d_{k}=\beta_{k}\exp{(\left(\rho_{0}-\rho_{k})/h_{0}\right)}V_{k}$ and $g_{k}=-Gm_{0}\rho_{k}^{-3}$ where $\beta_{k}=\beta_{0}\exp{([\mathbf{x}_{k}]_{5})}$, $V_{k}=\sqrt{[\mathbf{x}_{k}]_{3}^{2}+[\mathbf{x}_{k}]_{4}^{2}}$ and $\rho_{k}=\sqrt{[\mathbf{x}_{k}]_{1}^{2}+[\mathbf{x}_{k}]_{2}^{2}}$ with $\rho_{0}$, $h_{0}$, $G$, $m_{0}$ and $\beta_{0}$ as constants. The radar's range and bearing measurements are
    $[\mathbf{y}_{k}]_{1}=\sqrt{([\mathbf{x}_{k}]_{1}-\rho_{0})^{2}+[\mathbf{x}_{k}]_{2}^{2}}+v_{1}$, and 
    $[\mathbf{y}_{k}]_{2}=\tan^{-1}{\left(\frac{[\mathbf{x}_{k}]_{2}}{[\mathbf{x}_{k}]_{1}-\rho_{0}}\right)}+v_{2}$, 
where $v_{1}$ and $v_{2}$ represent the zero-mean measurement noise \cite{lu2011two}. 

For the inverse filter, we consider a linear observation 
    $\mathbf{a}_{k}=\left[[\hat{\mathbf{x}}_{k}]_{1},[\hat{\mathbf{x}}_{k}]_{2}\right]^{T}+\bm{\epsilon}_{k}$, 
where $\bm{\epsilon}_{k}\sim\mathcal{N}(\mathbf{0},3\mathbf{I}_{2})$. The initial state was $\mathbf{x}_{0}=[6500.4,349.14,-1.8093,-6.7967,0.6932]^{T}$. The initial state estimate $\Hat{\Hat{\mathbf{x}}}_{k}$ for IUKF was set to actual $\mathbf{x}_{0}$ with initial covariance estimate $\overline{\bm{\Sigma}}_{0}=diag(10^{-5},10^{-5},10^{-5},10^{-5},1)$. For forward UKF, $\kappa$ was chosen as $2.5$ such that the weight for $0$-th sigma point at $\hat{\mathbf{x}}_{k}$ is $1/3$ and all other sigma points have equal weights. Similarly, $\overline{\kappa}$ of IUKF was set to $3.5$. 

Fig. \ref{fig:UKF sim}b shows the (root) time-averaged error (over 100 runs) in position estimation and its RCRLB (also, time-averaged) for forward and inverse UKF (IUKF-1), including forward EKF and IUKF-2 (which incorrectly assumes the forward filter to be UKF when the adversary's actual forward filter is EKF). 
Here, the RCRLB is computed as $\sqrt{[\mathbf{J}^{-1}]_{1,1}+[\mathbf{J}^{-1}]_{2,2}}$ with $\mathbf{J}$ as the corresponding information matrix. 
The IUKF's error is lower than that of forward UKF, as is the case with their corresponding RCRLBs. Further, incorrect forward filter assumption (IUKF-2 case) does not affect the IUKF's estimation accuracy because forward UKF and EKF have similar estimation error. 
For the vehicle re-entry example, the IEKF's error (IEKF-1 and IEKF-2 cases of FM demodulation
)  was similar to IUKF and hence, omitted in Fig. \ref{fig:UKF sim}b. 

\section{Summary}
\label{sec:summary}
With the proposed IUKF, the defender is able to learn an estimate of the adversary’s inference by observing the latter’s actions. The IUKF is guaranteed to be stochastically stable if the forward UKF is stable under mild conditions on the system model. The proposed IUKF assumes that the adversary employs a forward UKF. However, our numerical experiments suggest that IUKF provides reasonable estimates and outperforms IEKF even when the former assumes an incorrect forward filter. When the system dynamics are not known to both agents, then this information may be learned through observations. In particular, RKHS-based function approximation coupled with EKF was employed to learn the unknown system parameters in \cite{singh2022inverse_part2}.

\bibliographystyle{IEEEtran}
\bibliography{main}
\end{document}